\documentclass[12pt, reqno]{amsart}
\usepackage{amsmath, amsthm, amscd, amsfonts, amsthm, amssymb, graphicx}
\usepackage[bookmarksnumbered, colorlinks, plainpages]{hyperref}

\input{mathrsfs.sty}

\newtheorem{theorem}{Theorem}[section]

\newtheorem{corollary}[theorem]{Corollary}
\theoremstyle{definition}

\numberwithin{equation}{section}
\newcommand{\la}{\left\langle}
\newcommand{\ra}{\right\rangle}

\begin{document}

\title{Bellman inequality for Hilbert space operators}

\author[A. Morassaei, F. Mirzapour, M.S. Moslehian]{A. Morassaei$^1$, F. Mirzapour$^1$ and M. S. Moslehian$^2$}
\vspace{-2cm}
\address{$^{1}$ Department of Mathematics, Faculty of Sciences, University of Zanjan, P. O. Box 45195-313, Zanjan, Iran.}

\email{morassaei@znu.ac.ir\\f.mirza@znu.ac.ir}

\address{$^{2}$ Department of Pure Mathematics, Center of Excellence in Analysis on Algebraic Structures (CEAAS), Ferdowsi University of Mashhad, P. O. Box 1159, Mashhad 91775, Iran.}

\email{moslehian@ferdowsi.um.ac.ir\\moslehian@ams.org}
\urladdr{\url{http://profsite.um.ac.ir/~moslehian/}}

\dedicatory{Dedicated to Professors A. Berman, M. Goldberg and R. Loewy}



\begin{abstract}
We establish some operator versions of Bellman's inequality. In particular, we prove that if $\Phi: \mathbb{B}(\mathscr{H}) \to \mathbb{B}(\mathscr{K})$ is a unital positive linear map, $A,B \in \mathbb{B}(\mathscr{H})$ are contractions, $p>1$ and $0 \leq \lambda \leq 1$, then
\begin{eqnarray*}
\big(\Phi(I_\mathscr{H}-A\nabla_{\lambda}B)\big)^{1/p}\ge\Phi\big((I_\mathscr{H}-A)^{1/p}\nabla_{\lambda}(I_\mathscr{H}-B)^{1/p}\big).
\end{eqnarray*}
\end{abstract} \maketitle

\section{Introduction and preliminaries}

\noindent
Bellman \cite{BB2} proved that if  $n, p$ are positive integers and $A,B,a_k,b_k$ $(1\le k\le n)$ are positive real numbers such that $A^p\ge\sum_{k=1}^na_k^p$ and $B^p\ge\sum_{k=1}^nb_k^p$, then
\begin{eqnarray}\label{m0}
\left(A^p-\sum_{k=1}^na_k^p\right)^{1/p}+\left(B^p-\sum_{k=1}^nb_k^p\right)^{1/p}\le
\left((A+B)^p-\sum_{k=1}^n(a_k+b_k)^p\right)^{1/p}
\end{eqnarray}
(see also \cite{BB1, MPF}). A ``multiplicative'' analogue of this inequality is due to J. Acz\'el, see \cite{Moslehian}. During the last decades several generalizations, refinements and applications of the Ballman inequality in various settings have been given and some results related to integral inequalities are presented. The interested reader is referred to \cite{1, 2, 3} and references therein. There is an area of research in which one seeks some operator or norm versions of some classical real inequalities. Such inequalities may hold for operators acting on a Hilbert space of finite or infinite dimensional. This is based on the fact that the self-adjoint operators (Hermitian matrices) can be regarded as a generalization of the real numbers. Among a lot of results concerning with operator or norm inequalities in the literature we would like to point out some of papers of M. Goldberg \cite{G1, G2}. In the present paper, we aim to establish some operator versions of the Bellman inequality involving the arithmetic means and positive linear maps and give some applications.

Throughout the paper, let $\mathbb{B}(\mathscr{H})$ denote the algebra of all bounded linear operators acting on a complex
Hilbert space $(\mathscr{H},\la \cdot,\cdot\ra)$, $\mathbb{B}_h(\mathscr{H})$ denote the algebra of all self-adjoint operators in $\mathbb{B}(\mathscr{H})$ and $I$ is the identity operator. In the case when
$\dim \mathscr{H} = n$, we identify $\mathbb{B}(\mathscr{H})$ with the full matrix algebra $\mathcal{M}_n(\mathbb{C})$ of all $n\times n$
matrices with entries in the complex field. An  operator $A\in\mathbb{B}_h(\mathscr{H})$ is called
positive (positive-semidefinite for matrices) if $\la Ax,x\ra\ge0$ holds for every $x\in \mathscr{H}$
and then we write $A\ge0$. For $A,B\in\mathbb{B}_h(\mathscr{H})$, we say $A\le B$ if $B-A \ge 0$. Let
$f$ be a continuous real valued function defined on an interval $J$. The function
$f$ is called operator decreasing if $B\le A$ implies $f(A)\le f(B)$ for all $A,B$ with
spectra in $J$. A function $f$ is said to be operator concave on $J$ if
$$
\lambda f(A) + (1 - \lambda)f(B)\le f(\lambda A + (1 - \lambda)B)
$$
for all $A,B\in\mathbb{B}_h(\mathscr{H})$ with spectra in $J$ and all $\lambda\in[0,1]$.
A map $\Phi:\mathbb{B}(\mathscr{H})\to\mathbb{B}(\mathscr{K})$ is called positive if $\Phi(A)\geq 0$ whenever $A\geq 0$ and is said to be normalized if $\Phi(I_\mathscr{H})=I_\mathscr{K}$.
We denote by $\mathbf{P}[\mathbb{B}(\mathscr{H}),\mathbb{B}(\mathscr{K})]$ the set of all positive linear maps $\Phi:\mathbb{B}(\mathscr{H})\to\mathbb{B}(\mathscr{K})$ and by $\mathbf{P}_N[\mathbb{B}(\mathscr{H}),\mathbb{B}(\mathscr{K})]$ the set of all normalized positive linear maps $\Phi\in\mathbf{P}[\mathbb{B}(\mathscr{H}),\mathbb{B}(\mathscr{K})]$. If $\Phi \in \mathbf{P}_N[\mathbb{B}(\mathscr{H}),\mathbb{B}(\mathscr{K})]$ and $f$ is an operator concave function on an interval $J$, then
\begin{equation}\label{jen}
f(\Phi(A))\ge\Phi(f(A))\,\quad\quad(\mbox{Davis-Choi-Jensen's Inequality})
\end{equation}
for every self-adjoint operator $A$ on $\mathscr{H}$, whose spectrum is contained in $J$, see \cite{msm}. If $\Phi \in \mathbf{P}_N[\mathbb{B}(\mathscr{H}),\mathbb{B}(\mathscr{K})]$, $0\le r\le1$ and $A$ is a positive invertible operator, then
\begin{equation}\label{pmi}
\Phi(A^r)\le\Phi(A)^r\quad\quad\quad(\mbox{H\"older-McCarthy type Inequality}).
\end{equation}
For more information see \cite[Chapters 1 and 5]{FMPS}.


\section{Main results}

We start this section with recalling that for positive operators $A, B$ and a real number $0 \leq \lambda \leq 1$, the arithmetic mean $A \nabla_\lambda B$ is defined by $(1-\lambda )A+\lambda B$, see \cite[Chapter V]{FMPS}. We denote $A\nabla_{1/2} B$ by $A\nabla B$. We first present a Bellman type operator inequality involving positive linear maps and operator concave functions.
\begin{theorem}\label{p1}
Suppose that $\Phi \in
\mathbf{P}_N[\mathbb{B}(\mathscr{H}),\mathbb{B}(\mathscr{K})]$, $f:
J\to(0,\infty)$ is an operator concave function, where $J$ is an
interval of $(0,\infty)$ and $A,B \in\mathbb{B}(\mathscr{H})$ are
positive operators with spectra contained in $J$. Then
\begin{eqnarray}\label{mp4}
f\big(\Phi(A)\nabla_{\lambda}\Phi(B)\big)\geq \Phi(f(A))\nabla_{\lambda}\Phi(f(B))\,\quad(0 \leq \lambda \leq 1).
\end{eqnarray}
\end{theorem}
\begin{proof}
Assume that $0 \leq \lambda \leq 1$, $f: J\to(0,\infty)$ is operator concave and $A,B \in\mathbb{B}(\mathscr{H})$. It follows from the linearity of $\Phi$ that
\begin{align*}
f\big(\Phi(A)\nabla_{\lambda}\Phi(B)\big)&=f\big(\Phi(A\nabla_{\lambda}B)\big)\\
&\ge\Phi\big(f(A\nabla_{\lambda}B)\big)\qquad\qquad\qquad\qquad\qquad\qquad\quad\qquad(\mbox{by~} \eqref{jen})\\
&\ge\Phi\big(f(A)\nabla_{\lambda}f(B)\big)\qquad\quad\quad(\mbox{by the operator concavity of~} f)\\
&=\Phi(f(A))\nabla_{\lambda}\Phi(f(B))\,.
\end{align*}
\end{proof}

\begin{corollary}[Operator Bellman Inequality]
If $\Phi\in \mathbf{P}_N[\mathbb{B}(\mathscr{H}),\mathbb{B}(\mathscr{K})]$, $A,B$ are contractions, $p>1$ and $0 \leq \lambda \leq 1$, then
\begin{eqnarray}\label{mp5}
\big(\Phi(I_\mathscr{H}-A\nabla_{\lambda}B)\big)^{1/p}\ge\Phi\big((I_\mathscr{H}-A)^{1/p}\nabla_{\lambda}(I_\mathscr{H}-B)^{1/p}\big).
\end{eqnarray}
\end{corollary}
\begin{proof}
Note that the function $g(t)=t^r$ is operator concave on $(0,\infty)$ if $0\le r\le1$ \cite[Corollary 1.16]{FMPS}, so is the function $f(t)=(1-t)^r$ on $(0,1)$ if $0\le r\le1$. It follows from the linearity and the normality of $\Phi$ that
\begin{align*}
\big(\Phi(I_\mathscr{H}-A\nabla_{\lambda}B)\big)^{1/p}&=\big(I_\mathscr{K}-\Phi(A)\nabla_{\lambda}\Phi(B)\big)^{1/p}\\
&\ge\big(I_\mathscr{K}-\Phi(A)\big)^{1/p}\nabla_{\lambda}\big(I_\mathscr{K}-\Phi(B)\big)^{1/p}\\
&\qquad\qquad\quad\quad\quad\quad\quad\quad\quad\quad(\mbox{by}~\eqref{mp4}~\mbox{for}~ f(t)=(1-t)^r)\\
&=\big(\Phi(I_\mathscr{H}-A)\big)^{1/p}\nabla_{\lambda}\big(\Phi(I_\mathscr{H}-B)\big)^{1/p}\\
&\ge\Phi\left((I_\mathscr{H}-A)^{1/p}\right)\nabla_{\lambda}\Phi\left((I_\mathscr{H}-B)^{1/p}\right)\qquad\quad(\mbox{by}~\eqref{pmi})\\
&=\Phi\big((I_\mathscr{H}-A)^{1/p}\nabla_{\lambda}(I_\mathscr{H}-B)^{1/p}\big)\,.
\end{align*}
\end{proof}

\begin{corollary}
Let $\Phi\in \mathbf{P}_N[\mathbb{B}(\mathscr{H}),\mathbb{B}(\mathscr{K})]$, $A,B$ be contractions and $0 \leq \lambda \leq 1$. Then
$$
\log\big(\Phi(A)\nabla_{\lambda}\Phi(B)\big)\ge\Phi(\log A)\nabla_{\lambda}\Phi(\log B)\,.
$$
\end{corollary}
\begin{proof}
Put $f(t)=\log t$ in Theorem \ref{p1}.
\end{proof}
.
\begin{corollary}
Let $0 \leq \lambda \leq 1$ and $A,B\in\mathbb{B}(\mathscr{H})$ be invertible positive operators. Then
\begin{eqnarray*}
(A\nabla_\lambda B)\log\left(A^{-1}!_\lambda B^{-1}\right)\ge\left(A\log A^{-1}\right)\nabla_\lambda\left(B\log B^{-1}\right)\,.
\end{eqnarray*}
\end{corollary}
\begin{proof}
Consider the identity map as $\Phi$ and put $f(t)=-t\log t$ in Theorem \ref{p1} and recall that $A!_\lambda B:=\left(A^{-1}\nabla_\lambda B^{-1}\right)^{-1}$.
\end{proof}

In the last result we establish Bellman's Inequality and a new variant of it.

\begin{theorem}
The following equivalent statements hold:
\begin{enumerate}
\item[(i)]If $n$ is a positive integer, $p>1$, $0 \leq \lambda \leq 1$ and $a_k,b_k$ $(1\le k\le n)$ are positive real numbers such that $1\ge\sum_{k=1}^na_k^p$ and $1\ge\sum_{k=1}^nb_k^p$, then
\begin{eqnarray}\label{mp3}
\left(1-\sum_{k=1}^na_k^p\right)^{1/p}\nabla_{\lambda}\left(1-\sum_{k=1}^nb_k^p\right)^{1/p}\le
\left(1-\sum_{k=1}^n\Big(a_k\nabla_{\lambda}b_k\Big)^p\right)^{1/p}\,.
\end{eqnarray}
\item[(ii)] [Classical Bellman Inequality] If If $n$ is a positive integer, $p>1$ and $M_1,M_2,a_k,b_k$ $(1\le k\le n)$ are positive real numbers such that $M_1^p\ge\sum_{k=1}^na_k^p$ and $M_2^p\ge\sum_{k=1}^nb_k^p$, then
{\footnotesize
\begin{eqnarray}\label{mp1}
\left(M_1^p-\sum_{k=1}^na_k^p\right)^{1/p}+\left(M_2^p-\sum_{k=1}^nb_k^p\right)^{1/p}\le
\left((M_1+M_2)^p-\sum_{k=1}^n(a_k+b_k)^p\right)^{1/p}\,.
\end{eqnarray}}
\end{enumerate}
\end{theorem}
\begin{proof}
Let $n, p$ be positive integers, $0 \leq \lambda \leq 1$ and $a_k,b_k$ $(1\le k\le n)$ be positive real numbers such that $1\ge\sum_{k=1}^na_k^p$ and $1\ge\sum_{k=1}^nb_k^p$. Set
$A=\left[\begin{array}{cc}\sum_{k=1}^na_k^p&0\\0&1\end{array}\right]\in \mathcal{M}_2(\mathbb{C})$, $B=\left[\begin{array}{cc}\sum_{k=1}^nb_k^p&0\\0&1\end{array}\right] \in \mathcal{M}_2(\mathbb{C})$. Then
\begin{align*}
\left(I-A\nabla_{\lambda}B\right)^{1/p}&=\left(\left[\begin{array}{cc}1&0\\0&1\end{array}\right]-
\left[\begin{array}{cc}(1-\lambda)\sum_{k=1}^na_k^p+\lambda\sum_{k=1}^nb_k^p&0\\0&1\end{array}\right]\right)^{1/p}\\
&=\left[\begin{array}{cc}\left((1-\sum_{k=1}^na_k^p)\nabla_{\lambda}(1-\sum_{k=1}^nb_k^p)\right)^{1/p}&0\\0&0\end{array}\right]\,.
\end{align*}
and
\begin{align*}
(I-A)^{1/p}\nabla_{\lambda}(I-B)^{1/p}&=(1-\lambda)
\left[\begin{array}{cc}1-\sum_{k=1}^na_k^p&0\\0&0\end{array}\right]^{\frac{1}{p}}+\lambda
\left[\begin{array}{cc}1-\sum_{k=1}^nb_k^p&0\\0&0\end{array}\right]^{\frac{1}{p}}\\
&=\left[\begin{array}{cc}(1-\sum_{k=1}^na_k^p)^{1/p}\nabla_{\lambda}(1-\sum_{k=1}^nb_k^p)^{1/p}&0\\0&0\end{array}\right]\,.
\end{align*}
It follows from \eqref{mp5} with the identity map $\Phi$ that
\begin{align*}
\left(1-\sum_{k=1}^na_k^p\right)^{1/p}\nabla_{\lambda}\left(1-\sum_{k=1}^nb_k^p\right)^{1/p}&\le
\left(\Big(1-\sum_{k=1}^na_k^p\Big)\nabla_{\lambda}\Big(1-\sum_{k=1}^nb_k^p\Big)\right)^{1/p}\\
&=\left(1-\Big(\sum_{k=1}^na_k^p\Big)\nabla_{\lambda}\Big(\sum_{k=1}^nb_k^p\Big)\right)^{1/p}\\
&=\left(1-\sum_{k=1}^n\Big(a_k^p\nabla_{\lambda}b_k^p\Big)\right)^{1/p}\\
&\le\left(1-\sum_{k=1}^n\Big(a_k\nabla_{\lambda}b_k\Big)^p\right)^{1/p}\\
&\quad\quad\mbox{(by the convexity of $t^p$ for $p>1$)}\,,
\end{align*}
which gives \eqref{mp3}.\\
(i)$\Rightarrow$(ii) Set $\lambda=\frac{M_2}{M_1+M_2}$ and replace $a_k, b_k$ by $a_k/M_1, b_k/M_2$, respectively, in \eqref{mp3} to get
\begin{align*}
&\frac{1}{M_1+M_2}\left(M_1^p-\sum_{k=1}^na_k^p\right)^{1/p}+\frac{1}{M_1+M_2}\left(M_2^p-\sum_{k=1}^nb_k^p\right)^{1/p}\\
&\quad=\left(1-\frac{M_2}{M_1+M_2}\right)\left(1-\sum_{k=1}^n\left(\frac{a_k}{M_1}\right)^p\right)^{1/p}+\frac{M_2}{M_1+M_2}
\left(1-\sum_{k=1}^n\left(\frac{b_k}{M_2}\right)^p\right)^{1/p}\\
&\quad\le\left(1-\sum_{k=1}^n\left(\left(1-\frac{M_2}{M_1+M_2}\right)\left(\frac{a_k}{M_1}\right)
+\left(\frac{M_2}{M_1+M_2}\right)\left(\frac{b_k}{M_2}\right)\right)^p\right)^{1/p}\\
&\quad=\frac{1}{M_1+M_2}\left((M_1+M_2)^p-\sum_{k=1}^n(a_k+b_k)^p\right)^{1/p}\,.
\end{align*}
We therefore deduce the desired inequality \eqref{mp1}.\\
(ii)$\Rightarrow$(i) Set $M_1=1-\lambda, M_2=\lambda$, and replace
$a_k$ and $b_k$ by $(1-\lambda)a_k$ and $\lambda b_k$, respectively,
in \eqref{mp1} to get \eqref{mp3}.
\end{proof}

\textbf{Acknowledgement.}{ The third author would like to thank ``Tusi Mathematical Research Group (TMRG)''}


\end{document}